\documentclass[11pt]{article}

\usepackage{amsmath}
\usepackage{amsfonts}
\usepackage{amssymb}
\usepackage{latexsym}
\usepackage{amsthm}

\newtheorem{theo}{Theorem}
\newtheorem{lemma}[theo]{Lemma}
\newtheorem{coro}[theo]{Corollary}
\newtheorem{prop}[theo]{Proposition}

\newtheorem{claim}[theo]{Claim}
\def\Z{\mathbb Z}
\def\Q{\mathbb Q}
\def\N{\mathbb N}
\def\R{\mathbb R}
\def\T{\mathcal T}

\begin{document}
\author{
K\'aroly J. B\"or\"oczky\thanks{Supported by
 the FP7 IEF grant GEOSUMSETS and OTKA 109789} 
 \and 
 Francisco Santos\thanks{Supported by the Spanish Ministry of Science (MICINN) through grant MTM2011-22792}
 \and 
 Oriol Serra \thanks{Supported by the Spanish Ministry of Science (MICINN) under
project MTM2011-28800-C02-01, and  by  the Catalan Research Council  under
grant  2009SGR1387}}

\title{On sumsets and convex hull}

\maketitle

\begin{abstract}
One classical result of Freimann gives the optimal lower bound for the cardinality of $A+A$ if
$A$ is a  $d$-dimensional finite set in $\R^d$.
Matolcsi and Ruzsa have recently generalized this lower bound to $|A+kB|$ if
$B$ is  $d$--dimensional, and $A$ is contained in the convex hull of $B$. We
characterize the equality case of the Matolcsi-Ruzsa bound. The argument is
based partially on understanding   triangulations of polytopes.
\end{abstract}

\section{Introduction}

The topic of this paper is the  cardinality of the sum of finite sets in the real affine space.
For thorough surveys and background, consult
I.Z. Ruzsa \cite{Ruz09}, and
T. Tao, V. Vu \cite{TaV06}.

A set $A$ in $\R^d$ is $d$--dimensional if it is not contained in any affine hyperplane.
One seminal result proved by G. Freiman \cite{Fre73}  is that for any finite $d$-dimensional  set $A$ in $\R^d$,
\begin{equation}
\label{Freiman}
|A+A|\geq (d+1)|A| - \frac{d(d+1)}2.
\end{equation}

This was recently generalized by M. Matolcsi and I.Z. Ruzsa \cite{MaR10} as follows.

\begin{theo}[Matolcsi-Ruzsa]
\label{AkB}
If $B$ is finite $d$-dimensional in $\R^d$ and $A\subset[B]$, then for every $k\in \N$
$$
|A+kB|\geq {d+k \choose k} |A| -  k {d+k \choose k+1}.
$$
\end{theo}

In particular, taking $A=B$ they get the following, of which
(\ref{Freiman}) is the case $k=2$:

\begin{coro}[Freiman-Matolcsi-Ruzsa]
\label{AkA}
If $A$ is finite $d$-dimensional in $\R^d$,  then for every $k\in \N$
$$
|kA|\geq {d+k-1 \choose k-1} |A| -  (k-1) {d+k-1 \choose k}.
$$
\end{coro}

In these results, for a set $X\subset\R^d$, we set $1X=X$, $kX=(k-1)X+X$
for $k\geq 2$, and $0X=\{0\}$. The sum $X+\emptyset$ is always the empty set. The convex hull of
the set $X\subset \R^d$ is denoted by $[X]$. Similarly,
$[x_1,\ldots,x_m]$ will denote the convex hull of  points
$x_1,\ldots,x_m\in \R^d$.

One of the motivations of the Matolcsi--Ruzsa inequality is the observation  that to prove
(\ref{Freiman}) for the sumset $A+A$,
 the relevant points  of the second summand are the vertices of $[A]$.

The goal of this paper is to give an explicit characterization of the sets $A$ and $B$ for which the inequality in Theorem~\ref{AkB} is tight, that is, for which $A\subset [B]$ and
\[
|A+kB|={d+k \choose k} |A| -  k {d+k \choose k+1}.
\]
We call such a pair $(A,B)$ a $k$--{\it critical} pair. As in M. Matolcsi and I.Z. Ruzsa \cite{MaR10},  triangulations
of $B$ have crucial role in our paper.
 By a triangulation $\T$ of $B$, we mean a triangulation  of $[B]$ where
the set of vertices of $\T$ is $B$. In addition, $\T$ is called  \emph{stacked} if it has 
$|B|-d$ full-dimensional simplices (which is the minimum possible number of simplices in a triangulation of
$|B|$ points in $\R^d$). 
As first steps in the characterization, we show that for every $k$--critical pair $(A,B)$:
\begin{itemize}
\item $B\subset A$ (Lemma~\ref{BinA}).
\item $(A\cap [B'], B')$ is also $k$--critical, for any subset $B'\subset B$ (Lemma~\ref{coro:subset-equa}).
\item  $B$ is \emph{totally stackable} (Corollary~\ref{stack}), meaning that all of its triangulations are stacked.
\end{itemize}

Total stackability is a very restrictive property that can be expressed in different ways (Lemma~\ref{lemma:stackable}) and totally stackable sets are completely characterized by B. Nill and A. Padrol (see Theorem~\ref{theo:stackablechar}). Section~\ref{sectriangs} includes these results and some preliminary background on triangulations.
The fact that $B$ needs to be totally stackable in order to have equality follows from the following refinement of Theorem~\ref{AkB} that we prove in Section~\ref{seccomb}.

\begin{theo}
\label{theo:h-vector-intro}
Let $\T$ be a shellable triangulation of $B$ with $h$-vector $(h_0,\dots,h_d)$. Let $A$ be such that $B\subset A \subset[B]$. Then,
\[
|A +kB| \ge
{d+k\choose k}|A| - k{d+k\choose k+1}  +
\sum_{j=2}^{\min(d,k+1)} h_j {d+k+1 - j \choose k+1 - j}.
\]
\end{theo}

The $h$-vector $(h_0,h_1,\dots,h_d)\in\N^{d+1}$ (here and in what follows $\N=\{0,1,2,\dots\}$) of a $d$-dimensional triangulation is a classical invariant in geometric combinatorics, which can be read either from the $f$-vector (the number of simplices of each dimension) or from a shelling.
See more background on this topic in Section~\ref{sectriangs}.
Since $h_i\ge 0$ for every $i$, Theorem~\ref{theo:h-vector-intro} implies Theorem~\ref{AkB}. But it also tells us that in order to have equality in Theorem~\ref{AkB} all the shellable triangulations of $B$ need to have $h_i=0$ for all $i\ge 2$, which is equivalent to them having $|B|-d$ simplices. Hence, $B$ needs to be totally stackable.

It is worth noticing that  the inequality in Theorem~\ref{theo:h-vector-intro} is equivalent to
\[
|A +kB| \ge {d+k\choose k}|A\setminus B| + \sum_{j=0}^{\min(d,k+1)} h_j {d+k+1 - j \choose k+1 - j}. 
\]
For the case $A=B$ this leads to the following refinement of Corollary~\ref{AkA}:

\begin{coro}
\label{coro:h-vector-intro}
Let $\T$ be a shellable triangulation of $A$ with $h$-vector $(h_0,\dots,h_d)$. Then,
\[
|kA| \ge
\sum_{j=0}^{\min(d,k)} h_j {d+k - j \choose k - j}. 
\]
\end{coro}

The geometric structure of critical pairs is complemented by its arithmetic structure. To express  this arithmetic structure we introduce the following concepts.
For finite $B\subset\R^d$, we write $\Lambda(B)$ to denote the additive subgroup of $\R^d$ generated by $B-B$, and hence by $B$ if $0\in B$. We note that $\Lambda(B)$ is called a lattice if it is of rank $d$, which will be the typical case. 
 We say that $A\subset [B]$ is {\it stable} with respect to  $B$, or $B$--{\em stable} if
$$
 (A+\Lambda(B))\cap [B] =A.
$$

The fact that $A$ is $B$--stable provides a substantial arithmetic structure to $A$. For example,
suppose that $A$ is $B$--stable and let $l$ be a line intersecting $A$ and such that $\Lambda(B)$ contains non-zero vectors parallel to $l$. Let $w$ be the shortest such vector (which is unique up to sign). Then $A\cap l$ can be partitioned into arithmetic progressions with common difference $w$, each of which equals $(x+\Z w)\cap [B]$ for some $x\in \R^d$. If, in addition, $l$ contains an edge $[u,v]$ of $[B]$, then one of these arithmetic progressions contains the vertices $u,v$ of the edge.
In particular, for two parallel lines $l,l'$ intersecting $A$ these arithmetic progressions have the same common difference ($w$ depends only on the direction of $l$) and if the lines contain edges $e, e'$ with $\ell (e)\ge \ell (e')$ of $B$ then  the translation of $A\cap e'$ within $e$ matching one vertex of $e$ is contained in $A\cap e$.    

With these geometric and arithmetic ingredients,  Sections~\ref{secsimplex},~\ref{secplanar} and~\ref{secprism} lead to the following explicit characterization of the critical pairs via a case study based in the  characterization of totally stackable sets.

\begin{theo}
\label{AkBequalist}
Let  $k\geq 1$, $d\geq 1$, and  let $A,B\subset \R^d$ be finite such that $B$ spans $\R^d$, and  $A\subset[B]$. Equality holds in Theorem~\ref{AkB} if and only if $B\subset A$,
$B$ is  contained in the union of the edges of $[B]$, and one of the following conditions hold.
\begin{description}

\item{(i)} $|B|=d+1$. That is, $B$ is the vertex set of a $d$-simplex.

\item{(ii)}  For $d\geq 1$, $B$ consists of the vertices of the simplex $[v_0,\ldots,v_d]$, 
 and some extra points on the edge $[v_0,v_d]$. 
The points of $B$ on this
edge are part of an arithmetic progression
$D$ contained in $A$, and $A\backslash(B\cup D)$
is the disjoint union of translates of $D\setminus \{v_0\}$.

\item{(iii)} For $d\ge 2$, $[B]$ is a prism with parallel vertical edges. $A$ is stable with respect to  $B$ and contained in the vertical edges of the prism.

 \item{(iv)}  For $d= 2$, $A$ consists of the vertices of a triangle  and the midpoints of its sides. 

\item{(v)} For $d= 2$, $[B]$ is a paralelogram, and $A$ is stable with respect to $B$ and contained in the boundary of $[B]$.

\item{(vi)} For $2\leq q<d$, $A$ and $B$ are the unions of some $d-q$ points and sets $A', B'$ respectively, where $(A',B')$ is  a pair of $q$-dimensional sets  of type  (iii), (iv) or (v).
\end{description}
\end{theo}

The characterization in Theorem~\ref{AkBequalist} reveals two interesting facts about critical pairs. 
\begin{itemize}
\item The characterization is independent of $k$. One direction (the fact that $k$--criticallity implies $(k-1)$--criticallity, if $k\ge 2$) is proved in Lemma~\ref{coro:k=1-equa}. The other direction is only proved as a consequence of the full characterization.

\item If $(A,B)$ is critical then $A$ is stable with respect to $B$. Actually, criticality of the pair $(A,B)$ depends on $A$ and the lattice $\Lambda (B)$ generated by the points of $B$ rather than the structure of $B$ itself. Again, without resorting to the full characterization, we only have a partial direct proof of this, namely the case of dimension one (Proposition~\ref{dimone}).
\end{itemize}

In turn, Theorem~\ref{AkBequalist} yields the following concerning the equality case of
 Corollary~\ref{AkA}.

 \begin{coro}
\label{AkAequalist}
Let  $k\geq 2$, $d\geq 2$, and  let the finite $A$  span $\R^d$. Equality holds in  Corollary~\ref{AkA} if and only if  one of the following conditions hold.
\begin{description}

\item{(i)} The set $A$ consists of the vertices of a simplex, and an arithmetic progression
contained in an edge of the simplex, starting and ending at the endpoints of the edge.

\item{(ii)} The set $A$ consists of the vertices of a simplex, and the midpoints of the sides of a  certain $2$-face of the simplex.

\item{(iii)}  For $d\geq q\geq 2$, $[A]$ is an iterated pyramid over a $q$-dimensional prism over a simplex  whose  vertical edges are parallel. There exists a non-zero  $w\in\R^d$
such that $A$ consists of the vertices of $[A]$, and for each vertical edge of the prism over a simplex,
 the arithmetic progression
of difference $w$ starting and ending at its endpoints.

\end{description}
\end{coro}

Actually,  Corollary~\ref{AkAequalist} has the concise form Corollary~\ref{AkAequaLambda}. To state this, we say that a triangulation $\T$ of a finite set $A$ spanning $\R^d$ is unimodular if $\Lambda(A)$ is a lattice with determinant $\Delta$, and
 each full dimensional simplex of $\T$ has volume $\Delta/d!$. We note that  if  
$A$ has a stacked unimodular triangulation, then all of its triangulations are unimodular and stacked.

 \begin{coro}
\label{AkAequaLambda}
Let  $k\geq 2$, $d\geq 2$, and  let the finite $A$  span $\R^d$. Equality holds in  Corollary~\ref{AkA} if and only if  
$A$ has a stacked unimodular triangulation.
\end{coro}

To prove Theorem~\ref{AkBequalist}, first we consider the one-dimensional case in 
Section~\ref{secdim1}, which is the base of the arithmetic structure of critical pairs. Next we discuss some useful properties of triangulations
of convex polytopes in Section~\ref{sectriangs}. Section~\ref{seccomb} reviews the proof of the Matolcsi-Ruzsa inequality Theorem~\ref{AkB}, and concludes with a technical, but useful, characterization (Theorem~\ref{AkBequa}) of the equality case.
Based on this result, we show in Section~\ref{secsufficiency} that 
the pairs $(A,B)$ listed in Theorem~\ref{AkBequalist} are $k$--critical for any $k\geq 1$. Theorem~\ref{AkBequa}
is also the base of the arguments leading to the fundamental properties
of $k$--critical  pairs in  Section~\ref{secbasic-critical}.
Finally, a case by case analysis in Sections~\ref{secsimplex}, \ref{secplanar} and \ref{secprism} 
describes explicitely the arithmetic structure of the cases in Theorem~\ref{AkBequalist}. In Section~\ref{secproof} we show how the
results of the previous sections imply that the list in Theorem~\ref{AkB} is complete.

\section{The case of dimension one}
\label{secdim1}

It is instructive to discuss the one-dimensional  version of Theorem~\ref{AkBequalist} first, because it does not require the geometric machinery built later on, and it provides the base of the arithmetic structure of  higher dimensional critical pairs.

For rational $0\leq b_1<\ldots<b_n$, $n\geq 2$, we define ${\rm gcd}\{b_1,\ldots,b_n\}$
to be the largest rational number $w$ such that $b_1/w,\ldots,b_n/w$ are integers. We observe that if $A,B\subset\R$ are finite such that $A\subset [B]= [0,1]$, then $A$
being stable with respect to $B$ is equivalent to saying  that $B\subset \mathbb{Q}$, and $A$ is the union of maximal arithmetic progressions in $[0,1]$ with difference  $w={\rm gcd} (B)$.

We note that the one--dimensional version of
Theorem~\ref{AkB} reads as follows. If $A,B\subset \R$ are finite sets with
$A\subset[B]$, and $k\geq 1$, then
\begin{equation}
\label{AkBdim1}
|A+kB|\geq (k+1)(|A|-1)+1.
\end{equation}

The first part of the next proposition gives the one--dimensional version of Theorem~\ref{AkBequalist}. The second part will be used later.
  
\begin{prop}
\label{dimone}
Let $k\geq 1$, and let $A,B\subset\R$ be finite such that $A\subset[B]= [0,1]$.
\begin{description}
\item{(i)} The pair $(A,B)$ is $k$--critical if and only if $\{ 0,1\}\subset A$ and $A$ is stable with respect to $B$.

\item{(ii)} If $C\subset(0,1)$ is finite, then
\begin{equation}
\label{CkBdim1}
|C+kB|\geq (k+1)|C|,
\end{equation}
with equality if and only if $C$  is stable with respect to $B$.
\end{description}
\end{prop}
\begin{proof}  If $A$ is stable with respect to $B$, then
$$
A+kB=A+k\{0,1\},
$$
and hence equality holds in \eqref{CkBdim1}, and also in \eqref{AkBdim1} 
provided that $\{ 0,1\}\subset A$.

We note that if either $0\not \in A$ or $1\not \in A$, then a translate of $A$ is contained in $(0,1)$,  a case dealt with  in (ii) which shows that $(A,B)$ is not $k$--critical. 
Thus let the pair $(A,B)$ be $k$--critical with $\{ 0,1\}\subset A$.  

If $B=\{0,1\}$, then  $A$ is clearly stable with respect to $B$.
Therefore we may assume that $|B|\geq 3$. We write $X'$ to denote the image of $X\subset\R$ in the torus $\R/\Z$ by the quotient map. In particular
\begin{equation}
\label{Aminus1}
|A'|=|A|-1.
\end{equation}
Let $\widetilde{A}\subset[0,2)$ be the set obtained by choosing the smallest element of $A+B$ in each  coset of $\Z$ intersecting $A+B$. 
 Since $0\in B$ yields that $\widetilde{A}\cap (A+1)=\emptyset$, the sum 
$A+kB$ contains 
 the disjoint union
$$
\{k+1\}\cup\widetilde{A}\cup ((A\backslash\{1\})+1) \cup\ldots\cup ((A\backslash\{1\})+k).
$$
We deduce using $|A'+B'|\geq |A'|$ that
\begin{equation}
\label{Aprimeineq}
|A+kB|\geq |A'+B'|+k|A'|+1\geq (k+1)|A'|-1.
\end{equation}
As  the pair $(A,B)$ is $k$--critical, (\ref{AkBdim1}) and (\ref{Aminus1}) yield that
$|A'+B'|= |A'|$.  In particular
$$
a+b\in A'\mbox{ \ for $a\in A'$ and $b\in B'$}.
$$
We deduce from $|B|\geq 3$  that there exists some non-zero element of $B'$,
which in turn implies by the finiteness of $A'$ that $B'$ generates a finite subgroup $H$ of $\R/\Z$, and $A'$ is the union of some cosets of $H$. It follows that $B\subset\Q$,
and $H$ is generated by $w'$ for $w={\rm gcd} (B)$.  This implies (i).
The argument for (ii) is completely analogous, only $k+1\not\in C+kB$, and hence
(\ref{Aprimeineq}) is replaced by $|C+kB|\geq |C'+B'|+k|C'|$ where $|C'|=|C|$.
\end{proof}

\section{Some observations about triangulations}
\label{sectriangs}

Throughout this paper, a triangulation of a finite point set $B\subset\R^d$ is a geometric simplicial complex with vertex set $B$ and underlying space $[B]$. A triangulation will be given as a list of $d$-simplices.

Let ${\T}=\{S_1,\dots, S_m\}$ be a triangulation of $B$. We say that the ordering $S_1,\dots, S_m$ of the simplices of $\T$ is a \emph{shelling} if, for every $i$, the intersection of $S_i$ with $S_1\cup\dots\cup S_{i-1}$ is a union of facets of $S_i$. Equivalently, if  $S_1\cup\dots\cup S_{i}$ is a topological ball for every $i$. The index of a simplex $S_i$ in a shelling is the number of facets of $S_i$ that are contained in $S_1\cup\dots\cup S_{i-1}$. That is, the index of $S_1$ is zero and the index of every other $S_i$ is an integer between 1 and $d$. The $h$-vector of a shelling is the vector $h=(h_0,\dots,h_d)$ with $h_i$ equal to the number of simplices of index $i$. We recall without proof some simple facts about shellings and $h$-vectors (see \cite[Section 9.5.2]{LRS10} or \cite[Chapter 8]{Zie94} for details):

\begin{lemma}
\label{lemma:shellings}
\begin{description}
\item{(i)} Not every triangulation is shellable, but every point set has shellable triangulations. For example, all \emph{regular triangulations} (which include \emph{placing}, \emph{pulling} and \emph{Delaunay} triangulations) are shellable.

\item{(ii)} The $h$-vector of a shellable triangulation is independent of the choice of shelling. In fact, the $h$-vector of a (perhaps non-shellable) triangulation can be defined as
\[
h_k = \sum_{i=0}^k  (-1)^{k-i} {d+1-i \choose k-i} f_{i-1},
\]
where $(f_{-1},\dots,f_d)$ is the $f$-vector of $\T$. That is, $f_i$ is the number of $i$-simplices in $\T$, with the convention that $f_{-1}=1$.

\item{(iii)} Every triangulation of $B$ has $h_0=1$, $h_1= |B|-d-1$, and $\sum h_i=m$, where $m$ and $d$ are the number of $d$-simplices and the dimension of $\T$.
\end{description}
\end{lemma}

One useful way of constructing triangulations of a point set is the \emph{placing} procedure, which is recursively defined as follows
(see~\cite[Section 4.3.1]{LRS10} for more details).
Let $B\subset\R^d$ be a finite point set and let $x\in B$ be such that $B':=B\setminus \{x\}$ is $d$-dimensional and $x\not\in[B']$. If $\T'$ is a
triangulation of  $B'$, we call \emph{placing of $x$ in $\T'$} the triangulation $\T$ of $B$ obtained adding to $\T'$ the pyramids with apex at $x$ of all
the boundary $(d-1)$-simplices of $\T'$ that are \emph{visible from $x$}. Here, we say that a $(d-1)$-simplex $S$ in the boundary of $B'$ is visible
from $x$ if its supporting hyperplane $H$ separates $x$ from $B'\setminus H$. Equivalently, if $[x,y]\cap [B'] = \{y\}$ for every point $y\in S$.
It can be shown that if $\T'$ is shellable then $\T$ is shellable too.

The placing procedure can be used to construct a (shellable) triangulation of $B$ from scratch, by choosing an initial simplex $S=[x_1,\dots,x_{d+1}]$ with $\{x_1,\dots,x_{d+1}\}\subset B$ and $S\cap (B\setminus \{x_1,\dots,x_{d+1}\}) = \emptyset$, or to extend a given triangulation of a subset $B'\subset B$ with $[B'] \cap (B\setminus B') = \emptyset$.

We observe that if $C=\{x_1,\ldots,x_{d+1}\}$ is affinely independent, 
$s\in\{1,\ldots,d+1\}$, and $t>0$, then for the facets $F_j=[C\backslash x_j]$ of $[C]$,
$j=1,\ldots,s$, we have
\begin{equation}
\label{relint}
t\cdot\left([C]\backslash  \left(\cup_{j=1}^s F_j\right)\right)=
\left\{\sum_{j=1}^{d+1}\lambda_jx_j:\,
\lambda_j>0 \mbox{ for $j\leq s$},\;\forall\lambda_j\geq 0,\;
\sum_{j=1}^{d+1}\lambda_j=t\right\}.
\end{equation}
Therefore if $k\geq 1$, and $S_1,\dots, S_m$ is a shelling of a triangulation $\T$, then
\begin{equation}
\label{TiSi}
T_i+kS_i = (k+1)T_i
\mbox{ \ for $i=2,\ldots,m$ and $T_i=S_i\backslash(S_1\cup\ldots\cup S_{i-1})$}.
\end{equation}

Of special interest for us will be stacked triangulations. A \emph{stacked triangulation} is one that satisfies any of the following equivalent properties, and they are a particular case of placing triangulations, hence shellable:

\begin{lemma}
\label{stackedequi}
The following properties are equivalent, for a triangulation $\T$ of a point set $B$.
\label{lemma:stacked}
\begin{description}
\item{(i)} The number of $d$-simplices in $\T$ equals $|B|-d$.
\item{(ii)} $h_i=0$ for all $i\ge 2$.
\item{(iii)} The dual graph of $\T$ is a tree. The dual graph is the graph having as vertices the $d$-simplices of $\T$ and as edges the adjacent pairs (pairs that share a facet).
\item{(iv)} Every simplex of dimension at most $d-2$ of $\T$ is contained in $\partial [B]$.
\end{description}
\end{lemma}

\begin{proof}
The equivalence of the first two properties follows from $\sum h_i=m$ and $h_0+h_1=|B| - d$.
For a shellable triangulation $\T$, the implications (ii)$\Rightarrow$(iii)$\Rightarrow$(iv)$\Rightarrow$(ii) are also trivial. Hence, the only thing we need to prove is that any  of (i), (iii) and (iv) implies $\T$ to be shellable. Let the simplices in $\T$ be ordered $S_1,\dots,S_m$ in such a way that $S_i$ shares at least one facet with $S_1\cup\dots\cup S_{i-1}$, which can always be done. Then:
\begin{description}
\item{(i)} $S_1\cup\dots\cup S_{i}$ has at most one vertex more than $S_1\cup\dots\cup S_{i-1}$. If the total number of vertices equals $|B|-d$ we need the number to always increase by one, which implies $(S_1\cup\dots\cup S_{i-1})$ intersects $S_i$ only in a facet.
\item{(iii)} If the dual graph is a tree, it has one less edge than vertices. Then, no $S_i$ has two facets in common with $S_1\cup\dots\cup S_{i-1}$. It may in principle have a facet plus some lower dimensional face $\sigma$, but this would imply the dual graph of the link of $\sigma$ in $S_1\cup\dots\cup S_{i}$ to become disconnected. Since at the end of the process all links have connected dual graphs, there has to be a $j>i$ such that $S_j$ also contains $\sigma$ and is glued to $S_1\cup\dots\cup S_{j-1}$ along at least two facets, a contradiction.
\item{(iv)} If every simplex of dimension at most $d-2$ of $\T$ is contained in $\partial [B]$, then every $(d-1)$-simplex in $\T$ disconnects $\T$. Hence the dual graph is a tree and, by the previous argument, $\T$ is shellable.
\end{description}
\end{proof}

We call a point set $B$ \emph{totally stackable} if all its triangulations are stacked. This poses heavy restrictions on the combinatorics of $B$, as we now see:

\begin{lemma}
\label{lemma:stackable}
Let $B\subset \R^d$ be a $d$-dimensional finite point set. The following conditions are equivalent:
\begin{description}
\item{(i)} $B$ is totally stackable.
\item{(ii)} Every $k$ points of $B$ lie in a face of $[B]$ of dimension at most $k$, for every $k$.
\item{(iii)} Every subset $C$ of at most $d-1$ points of $B$ has $[C]\subset \partial [B]$.
\end{description}
\end{lemma}

\begin{proof}
The implication $(ii)\Rightarrow(iii)$ is obvious, and $(iii)$ clearly implies the last property of  Lemma~\ref{lemma:stacked} for every triangulation, hence it implies $(i)$. So, we only need to show $(i)\Rightarrow(ii)$.

Let $C\subset B$ be a set of $k$ points and let $F$ be the minimal face of $[B]$ containing $C$
(the \emph{carrier} of $C$). Assume that $\dim(C) > k$ and, without loss of generality, that $C$ is affinely independent. It is easy to
show that $B_F:=B\cap F$ has a triangulation $\T_F$ using $C$ as a simplex. Since $[C]$ goes through the interior of $F$, the link of
$C$ in $\T_F$ is a $(\dim(C)-k)$--sphere. In particular, since $\dim(C)-k>0$, its dual graph has cycles. This $\T_F$ can be extended to a
triangulation of $B$ (for example via the placing procedure, see \cite[Section 4.3.1]{LRS10}) which will still have cycles in its dual graph.
\end{proof}

Properties (ii) and (iii) have the following straightforward consequences.
, which will be useful in order to give an explicit description of all possible totally stackable sets:
\begin{itemize}
\item If $B$ is totally stackable, every point of $B$ is either a vertex of $[B]$ or lies in the relative interior of an edge of $[B]$. That is,
$B$ is contained in the union of edges of $[B]$. We call the edges of $[B]$ that contain points of $B$ other than vertices \emph{loaded}.
\item Every subset $B'$ of a totally stackable set $B$ is totally stackable in $\operatorname{aff}(B')$.
\end{itemize}

Sets satisfying property (iii) of Lemma~\ref{lemma:stackable} are called \emph{of combinatorial degree one} by B. Nill, A. Padrol \cite{NiP}, who give a complete classification of them. The description uses iterated pyramids, which we define in terms of the join operator.
Let $B_1$ and $B_2$ be two finite sets in $\R^d$ whose affine hulls are of dimensions $d_1$ and $d_2$, respectively. We say that $B_1\cup B_2$ is a \emph{join} of $B_1$ and $B_2$ if the affine hull
of $B_1\cup B_2$ is of dimension $d_1+d_2+1$. For $i=1,2$, consider a triangulation for $B_i$ where the number of $d_i$-simplices is $m_i$. These two triangulations induce a triangulation for the join $B_1\cup B_2$ where the number of
$(d_1+d_2+1)$-simplices is $m_1m_2$. Moreover, all triangulations of a join arise in this way. The special case of a join where $B_2$ is a single point is called a \emph{pyramid}, and if $B_2$ is affinely independent it is an \emph{iterated pyramid} over $B_1$.

\begin{theo}[B.~Nill and A.~Padrol~\cite{NiP}]
\label{splitchar}
\label{theo:stackablechar}
Let $B$ be a finite set in $\R^d$ not contained in a hyperplane. Then $B$ is
totally stackable
  if and only if $B$ is contained in the union of the edges of $[B]$, and either of the following conditions holds.
\begin{description}
\item{(i)} $[B]$ is a simplex, and all loaded edges meet at a vertex.
\item{(ii)} $[B]$ is an iterated pyramid over a polygon, and every loaded edge  is a side of the polygon.
\item{(iii)} $[B]$ is an iterated pyramid over a  prism over a simplex,
and every loaded edge  is a vertical edge of the prism.
\end{description}
\end{theo}

Observe that $[B]$ can be a simplex also in case (ii).

\section{A proof of Theorem~\ref{AkB} and some consequences for critical pairs}
\label{seccomb}

In this section, we review the proof of Theorem~\ref{AkB}
from  M. Matolcsi and I.Z. Ruzsa \cite{MaR10} in order to analyze the equality case. This will lead to a technical but useful characterization of critical pairs (Theorem~\ref{AkBequa}),
a strengthening of the Matolcsi--Ruzsa inequality (Theorem~\ref{theo:h-vector}), and 
various fundamental properties of critical pairs (Theorem~\ref{equacond}). 

Recall that $kM=M+\dots +M$ denotes the $k$-fold Minkowski sum of $M$ with itself and $k\cdot M=\{kx: x\in M\}$ denotes the dilation of $M$ by a factor of $k$.
We note that if $M$ is a convex set in $\R^d$, and $k\geq 1$ is an integer, then
\begin{equation}
\label{kM}
kM=\{kx:\,x\in M\}=k\cdot M.
\end{equation}

The following simple observation will be often used.

\begin{claim}
\label{Matolcsi-Ruzsa-1}
Let $S=[C]$ be a $d$-simplex for $C= \{v_0,\ldots,v_d\}$.
\begin{description}
\item{ (i)} $|kC|={d+k \choose k}$.
\item{(ii)} For distinct points $a,b\in S$ and $k\ge 1$ we have
$$
(a+kC)\cap (b+kC)=\emptyset,
$$
unless both $a,b\in C$.
\end{description}
\end{claim}
\begin{proof} We may assume that $v_0$ is the origin, and $v_1,\ldots,v_d$ form the orthonormal basis.
In this case we have
$$
kC=\left\{(t_1,\ldots,t_d)\in\N^d:\,\sum_{i=0}^dt_i\leq k\right\},
$$
and hence (i) follows by enumeration.

For (ii), we observe that, for each pair $x,y\in kC$ of distinct points, the sets $x+[0,1)^d$ and $y+[0,1)^d$ are disjoint. Since $S\backslash\{v_0,\ldots,v_d\}\subset[0,1)^d$,  it follows from $(a+kC)\cap (b+kC)\neq\emptyset$ that either $a$ or $b$
is a vertex, say $a$ is a vertex of $S$. Then $x+a\in\Z^d$, thus $b$ is a vertex of $S$ as well.
\end{proof}

\begin{coro}
\label{coro:Bsimplex}
If $C= \{v_0,\ldots,v_d\}$ is the vertex set of a simplex and if  $A\subset[C]$, then
\[
|A + kC| = {d+k \choose k}|A|  - \sum_{i=1}^{|A\cap C|} \left( {d+k \choose k}  - {d+k+1-i \choose k} \right).
\]
In particular,
\[
|A + kC| \ge
{d+k \choose k}|A| - k{d+k \choose k+1},
\]
with equality if and only if $C\subset A$.
\end{coro}

\begin{proof}
Clearly,
\[
A + kC=\bigcup_{a\in A} (a+kC).
\]
By Claim~\ref{Matolcsi-Ruzsa-1}(i) each of the sets $a+kC$ has cardinality ${d+k \choose k}$ and, by Claim~\ref{Matolcsi-Ruzsa-1}(ii), they are pairwise disjoint except when $a,a'\in C\cap A$. That is:
\[
|A + kC|={d+k \choose k}|A \setminus C|  + |(A\cap C) + kC|.
\]
To prove (i) we only need to check that
\[
|(A\cap C) + kC| = \sum_{i=1}^{|A\cap C|} {d+k+1-i \choose k}.
\]
For this assume, as in the proof of Claim~\ref{Matolcsi-Ruzsa-1}(i), that
$v_0$ is the origin, and $v_1,\ldots,v_d$ form an orthonormal basis. Let $C_{l}:=\{v_0,v_1,\dots,v_{l-1}\}$, $1\le l\le d+1$,   and assume that $A\cap C=C_t$ for some $t$. Observe that
\[
C_1 +kC=kC,
\]
and, for each  $l=2,\dots,d+1$,
\[
(C_{l}+kC) \setminus (C_{l-1}+kC)=\left\{(t_1,\ldots,t_d)\in\N^d:\,
  \begin{array}{l}
    \sum_{i=0}^dt_i = k+1  \\ t_1=\dots=t_{l-2}=0 \\  t_{l-1}>0
  \end{array}
\right\}.
\]
Hence,
\[
|(C_{l}+kC) \setminus (C_{l-1}+kC)|={d+k+1-l \choose k},
\]
and
\[
|(A\cap C) + kC| = |C_t + kC| = \sum_{i=1}^{|A\cap C|}{d+k+1-i \choose k}.
\]

For the second part of the statement, observe that each summand in $\sum_{i=1}^{|A\cap C|} \left( {d+k \choose k}  - {d+k+1-i \choose k} \right)$ is non-negative, so
\[
 {d+k \choose k}|A| - |A + kC|  \le
 \sum_{i=1}^{d+1} \left( {d+k \choose k}  - {d+k+1-i \choose k} \right) =
 k{d+k \choose k+1},
\]
with equality if and only if $|A\cap C|=d+1$.
\end{proof}

\begin{proof}[\bf Proof of Theorem~\ref{AkB}]
Let $S_1,\ldots,S_m$ be a shelling of a triangulation $\T$ of $B$.
Let $C_i$ be the set of vertices $S_i$.
According to (\ref{kM}),
\begin{equation}
\label{k+1S}
\mbox{ \ $(k+1)S_i$, $i=1,\ldots,m$, form a triangulation of $(k+1)[B]$.}
\end{equation}

We define
\begin{eqnarray*}
A_1&=&A\cap S_1\\
A_i&=&A\cap \left(S_i\backslash (S_1\cup\ldots\cup S_{i-1})\right)
\mbox{ \ for $i=2,\ldots,m$}.
\end{eqnarray*}
We observe that $A_1,\ldots,A_m$ form a partition of $A$. Moreover, by shellability,
\begin{equation}
\label{Ai}
\mbox{$A_i\subset S_i$, and $A_i$ contains at most one vertex
of $S_i$ for $i=2,\ldots,m$.}
\end{equation}
We deduce from \eqref{TiSi} that
\begin{equation}
\label{AkBinequa}
|A+kB|\geq \sum_{i=1}^m|A_i+kC_i|.
\end{equation}

Now Corollary~\ref{coro:Bsimplex} yields that
\begin{equation}
\label{A1B1inequa}
|A_1+kC_1|\ge{d+k \choose k}|A_1| - k{d+k\choose k+1}
\end{equation}
(with equality if and only if $C_1\subset A_1$),
and Claim~\ref{Matolcsi-Ruzsa-1} (i) and (ii) imply by \eqref{Ai} that
\begin{equation}
\label{AkB1equa}
|A_i+kC_i|=\sum_{a\in A_i} |a+kC_i|={d+k \choose k}|A_i|
\mbox{ \  for $i=2,\ldots,m$.}
\end{equation}

Theorem~\ref{AkB} follows from combining \eqref{AkBinequa}, \eqref{A1B1inequa}  and \eqref{AkB1equa}.
\end{proof}

Using the notation of the above proof, the following characterization of equality in Theorem~\ref{AkB} follows from \eqref{TiSi} and \eqref{k+1S} on the one hand, and
\eqref{AkBinequa}, \eqref{A1B1inequa} and \eqref{AkB1equa} on the other hand.

\begin{theo}
\label{AkBequa}
Let $A, B\subset \R^d$ finite such that $A \subset [B]$ and ${\rm dim}[B]=d$, and let $k\ge 1$. The pair $(A,B)$ is $k$--critical if and only if for some shelling $S_1,\ldots,S_m$ of an arbitrary triangulation $\T$ of $B$, we have
\begin{description}
\item{(i)} $C_1\subset A$;
\item{(ii)} $A_i+kC_i=(A+kB)\cap (k+1)T_i$ for $i=1,\ldots,m$
\end{description}
where $C_i$ denotes the set of vertices $S_i$, 
$T_1=S_1$ and $T_i=S_i\backslash (S_1\cup\ldots\cup S_{i-1})$, $i\ge 2$, and
$A_i=A\cap T_i$ for $i=1,\ldots,m$.
\end{theo}

We will also use the following consequence of the proof of Theorem~\ref{AkB}.

\begin{lemma}
\label{lemAkBinside} 
Suppose that $A+x\subset{\rm int}[B]$ for some $x\in\R^d$. Then
\begin{equation}
\label{AkBinside}
|A+kB|\geq {d+k \choose k}|A|.
\end{equation}
\end{lemma}

\begin{proof} If  $A+x\subset{\rm int}[B]$ for some $x\in\R^d$, then $A\cap B=\emptyset$ can be assumed, and hence Corollary~\ref{coro:Bsimplex} gives, in the notation of the proof of Theorem~\ref{AkB}, $|A_1+kC_1|\ge{d+k \choose k}|A_1|$. Therefore \eqref{AkBinequa} and \eqref{AkB1equa} yields \eqref{AkBinside}.
\end{proof}

\section{Proof of sufficiency in Theorem~\ref{AkBequalist}}
\label{secsufficiency}

Based on Theorem~\ref{AkBequa}, we show that 
the pairs $(A,B)$  in Theorem~\ref{AkBequalist} are $k$--critical for any $k\geq 1$. 
First we show that we can restrict $B$ to the vertices of $[B]$ in the case of the pairs
$(A,B)$ listed  in Theorem~\ref{AkBequalist}.

\begin{lemma} 
\label{Bvertex}
If $(A,B)$ is any of the  pairs listed  in  Theorem~\ref{AkBequalist}, and $B'$ is the vertex set of $[B]$, then
$$
A+kB=A+kB'\mbox{ \ for any $k\geq 1$}.
$$
\end{lemma}

\begin{proof} It is sufficient to  prove that for any $a\in A$ and $b\in B$, there exist $a'\in A$ and $b'\in B'$ such that $a+b=a'+b'$. Since $B\subset A$, we may assume that $a,b\not\in B'$. The fact that $A$ is stable with respect to $B$ and the conditions in (ii)-(v) of Theorem \ref{AkBequalist} mean that $b$ belongs to an arithmetic progression $D$ along an edge $[u,v]$ of $[B]$  containing its two vertices $u,v$,  and $a$ belongs to some translate $x+D$ of this arithmetic progression. The result follows since 
$D+D=\{ u,v\}+D$. 
\end{proof}

\begin{proof}[Proof of necessity in Theorem~\ref{AkBequalist}]
If $[B]$ is a simplex then combining Lemma~\ref{Bvertex} and Corollary~\ref{coro:Bsimplex} yields equality in 
Theorem~\ref{AkB}. 

Therefore we assume that $[B]$ is an iterated pyramid over a $q$-dimensional prism $P$, 
$2\leq q\leq d$,
and referring to Lemma~\ref{Bvertex}, also that $B$ consists of  the vertices of $[B]$. Let $B_0= B\backslash(B\cap P)$. We write $v_1,\ldots,v_q,w_1,\ldots,w_q$ to denote the vertices of $P$ in a way such that the vectors $w_i-v_i$ are parallel pointing into the same direction for $i=1,\ldots,q$. We define $S_i=[\{v_1,\ldots,v_i, w_i\ldots w_q\}\cup B_0]$ for $i=1,\ldots,q$,
and hence $S_1,\ldots,S_q$ form a shelling of the corresponding triangulation of $B$.
We write $A_i$, $C_i$, $T_i$ to denote the corresponding sets defined in Theorem~\ref{AkBequa} for $i=1,\ldots,q$.

Let $k\geq 1$ and $i\in \{1,\ldots,q\}$. We claim that assuming $v_i=0$, we have
\begin{eqnarray}
\label{prismlatticeAB}
A+kB&=&(A_i+\Lambda(B))\cap(k+1) [B] \\
\label{prismlatticeAiCi}
A_i+kC_i&=&(A_i+\Lambda(B))\cap(k+1)T_i .
\end{eqnarray}

Before proving (\ref{prismlatticeAB}) and (\ref{prismlatticeAiCi}), we point out that
they readily yield Theorem~\ref{AkBequalist} (ii) for the shelling $S_1,\ldots,S_q$.
Since Theorem~\ref{AkBequalist} (i) holds by $B\subset A$, we deduce 
equality in Theorem~\ref{AkB} for the pair $(A,B)$.

To verify (\ref{prismlatticeAB}) and (\ref{prismlatticeAiCi}), we observe that each coset of $\Lambda(B)$  intersecting $A$  has a representative in $A_i$. In other words,
\begin{equation}
\label{ALambdaBAi}
A+\Lambda(B)=A_i+\Lambda(B).
\end{equation}

We distinguish two cases. If  $P$ is parallelogram, then $P$ is actually a fundamental parallelogram for the two-lattice
$\Lambda(B)\cap {\rm lin}P$. Since $A$ is stable with respect to $\Lambda(B)$,
we deduce (\ref{prismlatticeAB}) by (\ref{ALambdaBAi}). In addition 
(\ref{prismlatticeAiCi}) follows from (\ref{TiSi}), and the fact that the non-zero elements of $C_i$ form a $\Z$-basis of $\Lambda(B)$.

Therefore we assume that $P$ is not a parallelogram, and hence $A\cap P$ is contained
in the vertical edges $[v_j,w_j]$ of $P$, $j=1,\ldots,q$. Let 
$$
\{z_1,\ldots,z_{d-1}\}=B_0\cup(\{v_1,\ldots,v_q\}\backslash\{v_i\}),
$$
thus there exists $w\in\R^d$ pointing into the same direction as $w_i-v_i$ such that
$\{w,z_1,\ldots,z_{d-1}\}$ form a $\Z$-basis for $\Lambda(B)$. It follows that
there exist integers $m_1,\ldots,m_q\geq 1$ such that
$w_j-v_j=m_jw$ for $j=1,\ldots,q$, and there exists
$\Omega\subset[0,1)$ such that $v_j+tw\in A$ for $j\in\{1,\ldots,q\}$
and $t\in[0,m_j]$ if and only if 
$t-\lfloor t \rfloor\in\Omega$. We define the integers $n_1,\ldots,n_{d-1}$ by
$n_p=m_j$ if $z_p=m_j$, and $n_p=0$ if $z_p\in B_0$. Writing
$$
\Xi=\{(i_1,\ldots,i_{d-1})\in\Z^{d-1}:\,i_1+\ldots+i_{d-1}=k\mbox{ and } i_j\geq 0\},
$$
we deduce (\ref{prismlatticeAB}) from (\ref{ALambdaBAi}) and
\begin{eqnarray*}
A+kB&=&\bigcup_{(i_1,\ldots,i_{d-1})\in\Xi}\left\{tw+\sum_{j=1}^{d-1}i_jz_j:\,
t-\lfloor t \rfloor\in\Omega\mbox{ and }
0\leq t\leq \sum_{j=1}^{d-1}i_jn_j\right\}\\
&=&(A_i+\Lambda(B))\cap(k+1) [B].
\end{eqnarray*}

Turning to (\ref{prismlatticeAiCi}), we observe that $C_i\backslash\{v_i\}$ form a basis for $\R^d$. Combining this fact with (\ref{TiSi}) yields
$$
A_i+kC_i=(A_i+\Lambda(C_i))\cap (k+1)T_i.
$$
Since $\{w\}\cup(C_i\backslash\{v_i,w_i\})$ form a $\Z$-basis for $\Lambda(B)$,
and $A_i+w\subset A_i+\Lambda(C_i)$, we deduce that
$A_i+\Lambda(C_i)=A_i+\Lambda(B)$. We conclude (\ref{prismlatticeAiCi}), and in turn that  equality holds in Theorem~\ref{AkB} for the pairs $(A,B)$  in  Theorem~\ref{AkBequalist}.
\end{proof}

\section{Basic properties of $k$--critical pairs}
\label{secbasic-critical}

The goal of the section is to prove Theorem~\ref{equacond} listing some fundamental properties
of $k$--critical pairs. The first one is a direct consequence of Theorem~\ref{AkBequa}.

\begin{lemma}
\label{BinA}
If $(A,B)$ is  a $k$--critical pair, $k\ge 1$,  then $B\subset A$.
\end{lemma}
\begin{proof} For any $x\in B$, we consider a shellable triangulation $\T$ with $x\in C_1$ for the first simplex $S_1=[C_1]$ of $\T$ (this can be achieved, for example, via a placing triangulation).
 Theorem~\ref{AkBequa} yields  $C_1\subset A$, thus $x\in A$.
\end{proof}

Based on Theorem~\ref{AkBequa}, we  prove that criticality is preserved by taking subsets of $B$.

\begin{lemma}
\label{coro:subset-equa}
Let $A, B$ be $d$-dimensional point sets with $A \subset [B]$, and let $k\ge 1$. If $(A,B)$ is $k$--critical, then $(A\cap [B'],B')$ is also $k$--critical for every  $B'\subset B$.
\end{lemma}

\begin{proof} Let $\widetilde{B}=B\cap[B']$ and $\widetilde{A}=A\cap[B']$. Since $B'\subset \widetilde{B}$ and 
$[B']= [\widetilde{B}]$,  it is sufficient to prove that $(\widetilde{A},\widetilde{B})$  is $k$--critical.

 If ${\rm dim}\,[B']=d$, then constructing a placing triangulation first for $\widetilde{B}$, we obtain some shelling $S_1,\ldots,S_m$ of a triangulation of $B$ such that the union of $S_1,\ldots,S_n$ is $[B']$ for some $n\leq m$. Now Theorem~\ref{AkBequa} (i) and (ii) for the pair $(A,B)$ readily yield the analogous properties for the pair $(\widetilde{A},\widetilde{B})$. 

Next we assume that ${\rm dim}[B']=q<d$. We choose $x_1,\ldots,x_{d-q}\in B$ such that for 
$B^*=\{x_1,\ldots,x_{d-q}\}\cup \widetilde{B}$, we have ${\rm dim}[B^*]=d$ and $B\cap [B^*]=B^*$. Let $A^*=A\cap [B^*]$. We observe that
$L={\rm aff}\,B'$ is a supporting $q$-plane to $[B^*]$, with $\widetilde{A}=A^*\cap L$, $\widetilde{B}=B^*\cap L$ and 
$[\widetilde{B}]=[B^*]\cap L$. 

Let $\widetilde{S}_1,\ldots,\widetilde{S}_n$ be a shelling of some triangulation of $\widetilde{B}$, and let
$\widetilde{A}_i$, $\widetilde{C}_i$, $\widetilde{T}_i$ for  $i=1,\ldots,n$ be the 
corresponding sets for Theorem~\ref{AkBequa}. We need to prove that they satisfy Theorem~\ref{AkBequa}.
Since $B\subset A$ according to Lemma~\ref{BinA},  Theorem~\ref{AkBequa} (i) readily follows, and all we have to verify is Theorem~\ref{AkBequa} (ii).

To achieve that, we observe that
$S_1,\ldots,S_n$ is a shelling of a triangulation of $B^*$ where
  $S_i=[x_1,\ldots,x_{d-q},\widetilde{S}_i]$ for $i=1,\ldots,n$. Writing $A_i$, $C_i$ and $T_i$ to denote the corresponding sets in
Theorem~\ref{AkBequa}, we have $A_1=\widetilde{A}_1\cup \{x_1,\ldots,x_{d-q}\}$,
 $A_i=\widetilde{A}_i$ for $i=2,\ldots,n$, moreover $\widetilde{C}_i=C_i\cap L$ and $\widetilde{T}_i=T_i\cap L$ 
for $i=1,\ldots,n$. The pair $(A^*,B^*)$  is $k$--critical by the  argument above because $B^*\subset B$ with ${\rm dim}B^*=d$, and
hence $A_i$, $C_i$ and $T_i$ satisfy Theorem~\ref{AkBequa} (ii). It follows that the
same conclusion holds for $\widetilde{A}_i$, $\widetilde{C}_i$ and $\widetilde{T}_i$, as for $i=1,\ldots,n$, we have
$$
(\widetilde{A}+k\widetilde{B})\cap \widetilde{T}_i=L \cap \left((A^*+k B^*)\cap T_i \right)
\subset L\cap (A_i+kC_i)=\widetilde{A}_i+k\widetilde{C}_i,
$$
where the first and the last equality is a consequence of the fact that $L$ is a supporting $q$-plane to $[B^*]$.
\end{proof}

Under the assumption $B\subset A$, we use the same ideas of the proof of Theorem~\ref{AkB} to obtain the following refinement of Theorem \ref{theo:h-vector-intro}:

\begin{theo}
\label{theo:h-vector}
Let $A,B\subset \R^d$ be finite such that ${\rm dim}\,[B]=d$ and $B\subset A \subset[B]$, and let $\T$ be a shellable triangulation of $B$ with $h$-vector $(h_0,\dots,h_d)$. Then
\begin{equation}
\label{eq:AkB-hvector}
|A +kB| \ge
{d+k\choose k}|A| - k{d+k\choose k+1}  +
\sum_{j=2}^{\min(d,k+1)} h_j {d+k+1 - j \choose k+1 - j}.
\end{equation}
\end{theo}

\begin{proof}
Let $S_1,\dots,S_m$   be a shelling of $\T$. We keep the same notation for
$A_i$, $C_i$ and $T_i$
as in the statement of Theorem~\ref{AkBequa}. In particular, we have 
\[
|A+kB|= \sum_{i=1}^m |(A+kB)\cap (k+1)T_i|.
\]

Let $A_i'=A_i\setminus C_i$. In each $(k+1)T_i$, we have, by the same argument as in Corollary~\ref{coro:Bsimplex}, and taking into account that $C_i\subset B \subset A$,
\begin{equation}
\label{AkCinequa}
|(A+kB)\cap (k+1)T_i|\ge |A_i+ kC_i| = |A'_i+ kC_i| + |(k+1)C_i \cap (k+1) T_i|.
\end{equation}

The first summand is 
\[
|A'_i+ kC_i| = {d+k\choose k}|A'_i|,
\]
by  Claim~\ref{Matolcsi-Ruzsa-1} (i), and so
\[
\sum_{i=1}^m |A'_i+ kC_i| = {d+k\choose k}|A\setminus B|.
\]

For the second summand, let $s_i$ be the index of $S_i$ in the shelling, and hence \eqref{relint} and enumeration yield 
\[
|(k+1)C_i \cap (k+1) T_i| =
\begin{cases}
 {d+k+1 - s_i \choose k+1 - s_i} & \text{if } s_i\le k+1, \\
0 &  \text{otherwise}.\\
\end{cases}
\]
Put differently,
\[
\sum_{i=1}^m |(k+1)C_i \cap (k+1) T_i| = \sum_{j=0}^{\min(d,k+1)} h_j {d+k+1 - j \choose k+1 - j}.
\]

Since $(k+1)C_i $ and $A'_i+ kC_i$ are disjoint by Claim~\ref{Matolcsi-Ruzsa-1} (ii),
we obtain
\begin{align*}
|A+kB|
\ge& \sum_{i=1}^m|A'_i+ kC_i| + \sum_{i=1}^m |(k+1)C_i \cap (k+1) T_i|\\
 = &
{d+k\choose k}|A\setminus B| + \sum_{j=0}^{\min(d,k+1)} h_j {d+k+1 - j \choose k+1 - j} \\
=&
{d+k\choose k}|A| + {d+k+1\choose k+1} -(d+1){d+k\choose k} +
\sum_{j=2}^{\min(d,k+1)} h_j {d+k+1 - j \choose k+1 - j}\\
=&
{d+k\choose k}|A| - k{d+k\choose k+1}  +
\sum_{j=2}^{\min(d,k+1)} h_j {d+k+1 - j \choose k+1 - j},
\end{align*}
where, in the last step, we use $h_0=1$ and $h_1=|B|-d-1$.
\end{proof}

\begin{coro}
\label{stack}
If the pair $(A,B)$ is $k$--critical with ${\rm dim}[B]=d$, then $B$ is totally stackable. In particular, $B$ is contained in the union of the edges of $[B]$.
\end{coro}

\begin{proof}
We have $B\subset A$ according to Lemma~\ref{BinA}. Theorem~\ref{theo:h-vector} yields that every shellable (in particular, every regular) triangulation has $h_2=0$. According to the characterization of the $h$-vectors 
by R.P. Stanley \cite{Sta77} (see also Theorem~8.34 in G.M. Ziegler \cite{Zie95}),
we have  $h_j=0$ for $j\ge 2$, which, by Lemma \ref{lemma:stacked}, implies that $B$ is stacked. That is, every regular triangulation of $B$ has $|B|-d$ $d$-simplices. It is a fact (see~\cite[Theorem 8.5.19]{LRS10}) that then all the triangulations (regular or not) have the same number $|B|-d$ of $d$-simplices. That is, $B$ is totally stackable.
\end{proof}

Next we prove that equality in Theorem~\ref{AkB} is preserved under reducing the value of $k$.

\begin{lemma}
\label{coro:k=1-equa}
 If $(A,B)$ is $k$--critical for $k\ge 2$ with ${\rm dim}[B]=d$, then it is also $k'$--critical for every $k'=1,\ldots,k-1$. 
\end{lemma}

\begin{proof}
Let $S_1,\dots,S_m$ be a shelling of a triangulation of $B$. We use the notation of Theorem~\ref{AkBequa}.
Condition (i) of Theorem~\ref{AkBequa} is independent of $k$, and hence we need to check condition (ii).

Let $z\in (A+(k-1)B)\cap T_i$ for $i=1,\ldots,m$, and 
what we need to show is that  $z\in A_i+(k-1)C_i$. Since $B$ is totally stackable by 
Corollary~\ref{stack}, we may assume that $C_i=\{v_0,\ldots,v_d\}$ in a way such that
$T_i=S_i\backslash [v_1,\ldots,v_d]$ if $i\geq 2$, and
$z\not\in [v_1,\ldots,v_d]$ if $i=1$.  In particular, $v_0\in A_i$, and 
$$
z=\sum_{j=0}^d\lambda_jv_j \mbox{ \ where $\lambda_0>0$, $\lambda_j\geq 0$ for $j=1,\ldots,d$,
and  $\lambda_0+\ldots+\lambda_d=k$}
$$
by (\ref{relint}). If $z\in kC_i$, or equivalently each $\lambda_i$ is an integer, then $v_0\in A_i$ and
$\lambda_0>0$ yield that $z\in A_i+(k-1)C_i$. Therefore we assume that $z\not\in kC_i$.

Since $z+v_0\in (A+kB) \cap (k+1)\,T_i$ and $(A,B)$ is $k$--critical, 
Theorem~\ref{AkBequa} (ii) yields that 
\begin{equation}
\label{za}
z+v_0=a+\sum_{j=0}^dm_jv_j
\end{equation}
 where $a\in A_i$, every $m_j\geq 0$ is an integer.
and  $m_0+\ldots+m_d=k$.
We have 
$$
a=\sum_{j=0}^d\alpha_jv_j \mbox{ \ where $\alpha_0>0$, $\alpha_j\geq 0$ for $j=1,\ldots,d$,
and  $\alpha_0+\ldots+\alpha_d=1$,}
$$
and $\alpha_0>0$ is a consequence of \eqref{relint}.
As $C_i$ is affinely independent, the coefficients satisfy $\lambda_0+1=\alpha_0+m_0$ and
$\lambda_j=\alpha_j+m_j$ for $j=1,\ldots,d$.

Since $z\not\in kC_i$, we deduce that some $\alpha_j$ is not integer, which in turn implies that $\alpha_0<1$. 
We have $\alpha_0+m_0=\lambda_0+1>1$ based on  (\ref{za}), thus
$m_0\geq 1$ by $\alpha_0<1$. Therefore
$$
z=a+(m_0-1)v_0+\sum_{j=1}^dm_jv_j\in A_i+(k-1)C_i,
$$
as it is required by Theorem~\ref{AkBequa} (ii).
\end{proof}

We summarize Lemmas~\ref{BinA}, \ref{coro:subset-equa} and \ref{coro:k=1-equa} and Corollary~\ref{stack} as follows.

\begin{theo}
\label{equacond}
 If the pair $(A,B)$ is $k$--critical for $k\ge 2$  with ${\rm dim}\,[B]=d$, then
\begin{description}
\item{(i)} $B\subset A$;
\item{(ii)} $(A\cap [B'],B')$ is also $k$--critical for every  $B'\subset B$;
\item{(iii)} $B$ is totally stackable, thus $B$ is contained in the union of the edges of $[B]$;
\item{(iv)} $(A,B)$ is  $1$--critical, and hence $|A+B|= (d+1)|A|  - d(d+1)/2$.
\end{description}
\end{theo}

From now on we consider $k$--critical pairs $(A,B)$ for $k=1$, which will be simply called {\it critical} pairs. Theorem~\ref{equacond} (iv) shows that $k$--critical pairs are critical. 
We also speak about critical sets in the case of the one dimensional version 
$|A+B|\geq 2|A|-1$ of the Matolcsi-Ruzsa inequality.

\section{The case of  a simplex}
\label{secsimplex}

In this section we consider the case where $[B]$ is   a $d$--simplex. First we discuss iterated pyramids, a case that will be used later on as well.

\begin{lemma}
\label{join}
Let $1\leq q< d$, and let $(A,B)$ be a critical pair with ${\rm dim}[B]=d$ such that
$[B]$ is an iterated pyramid over $[B_0]$ where $B_0\subset B$ and ${\rm dim}[B_0]=q$.
Then \begin{enumerate}
\item[(i)] $(A\cap [B_0],B_0)$  is a critical pair, and
\item[(ii)] $|(A\cap L)+B_0|=(q+1)|A\cap L|$ for any affine $q$-plane $L$ 
parallel to $L_0={\rm aff} (B_0)$  intersecting $A$ and avoiding the vertices of $[B]$.
\end{enumerate}
\end{lemma}
\begin{proof} We have $B\subset A$ by Theorem~\ref{equacond} (i), and 
the pair $(A\cap [B_0],B_0)$  is  critical by Theorem~\ref{equacond} (ii).
Let $[B]=[x_1,\ldots,x_{d-q},B_0]$, and let
$\widetilde{B}$ be the join of $\{x_1,\ldots,x_{d-q}\}$ and $B_0$.
In particular, $(A,\widetilde{B})$ is a critical pair, by Lemma~\ref{coro:subset-equa}.

We may assume that $0\in B_0$. We divide $A$ into equivalence classes
according to the cosets of $H=\Z x_1+\ldots+\Z x_{d-q}+L_0$, and hence adding $\widetilde{B}\subset H$
to different equivalence classes results in disjoint sets. One equivalent class
is $\widetilde{A}=\{x_1,\ldots,x_{d-q}\}\cup (A\cap [B_0])$, and
\begin{equation}
\label{A'iterated}
|\widetilde{A}+\widetilde{B}|\geq (d+1)|\widetilde{A}|-d(d+1)/2
\end{equation}
by Theorem~\ref{AkB}.  Any other equivalence class is of the form
$A\cap L$ for an affine $q$-plane $L$ 
parallel to $L_0$ that  avoids $\widetilde{B}$ and intersects $A$. Since a translate of $A\cap L$ is contained in the relative interior of $[B_0]$, 
Lemma~\ref{lemAkBinside}  yields
 \begin{equation}
\label{ALiterated}
|(A\cap L)+\widetilde{B}|\geq (d+1)|A\cap L|.
\end{equation}
As $(A,\widetilde{B})$ is a critical pair, (\ref{A'iterated}) and (\ref{ALiterated}) imply
$$
(d+1)|A\cap L|=|(A\cap L)+\widetilde{B}|=|(A\cap L)+B_0|+\sum_{i=1}^{d-q}|(A\cap L)+x_i|,
$$
therefore $|(A\cap L)+B_0|=(q+1)|A\cap L|$.
\end{proof}

We recall that an edge of $[B]$ is loaded if it contains at least three points of $B$.

\begin{prop}
\label{simplex}
Let $A,B$ be finite $d$--dimensional sets in  $\R^d$, $d\geq 2$  with  $A\subset[B]$. If  $[B]$ is a simplex, and
$(A,B)$ is a critical pair, then $B\subset A$, $B$ is contained in the edges of $[B]$, and one of the following conditions hold:
\begin{description}
\item{(i)}   $|B|=d+1$; or

\item{(ii)} there is a unique  loaded edge $[u,v]$ of $B$; the points of $B$ in this edge are part of an arithmetic progression $D$ contained in $A$; and $A\backslash(B\cup D)$
is the disjoint union of translates $D\setminus \{v\}$; or

\item{(iii)}  there exist two or three loaded edges for $B$, which are sides of a two dimensional face $T$ of $[B]$, and $A$ consists of the vertices of $[B]$, and the midpoints of the sides of $T$.
\end{description}
\end{prop}
\begin{proof} The facts that
$B\subset A$ and $B$ is contained in the edges of $[B]$ follow from Theorem~\ref{equacond} (i) and (iii). Let $B'$ be the vertex set of $[B]$.

We may assume $|B|>d+1$, and hence there exists a loaded edge $[u,v]$ of $[B]$.

It follows from  Lemma~\ref{join} (i)  and  Proposition~\ref{dimone} (i)
that  $B\cap [u,v]$ is part of a maximal arithmetic progression $D$
containing the vertices $u,v$,  and $A\cap [u,v]$ contains $D$, and the rest of $A\cap [u,v]$
is the disjoint union of translates of $D\setminus \{v\}$. In particular,
Lemma~\ref{join} (ii)  and  Proposition~\ref{dimone} (ii) imply that
\begin{equation}
\label{uvD}
\mbox{ $A\backslash (D\cup B')$ is the disjoint union of translates of $D\setminus \{v\}$.}
\end{equation}

If  there exists a unique loaded edge of $B$, then \eqref{uvD} yields (ii).
Therefore we may assume that there are at least two loaded edges of $[B]$. Since $B$ is totally stackable by Theorem~\ref{equacond} (iii),  it follows that either all loaded edges meet in a vertex, or they form a triangular $2$--dimensional face by Theorem~\ref{theo:stackablechar}. 

Therefore we may assume $[v_0,v_1]$ and $[v_0,v_2]$ are two loaded edges of $B$ with $v_0=0$, and let  $T=[v_0,v_1,v_2]$ be the $2$-face containing these two edges.
In particular, $(A\cap T, B\cap T)$ is a critical pair
by Theorem~\ref{equacond} (ii). It follows by \eqref{uvD} that for $i=1,2$, $A\cap T$ contains an arithmetic progression $D_i$ of length  $m_i\geq 3$ with endpoints $v_0$ and $v_i$.
According to \eqref{uvD}, $a_1=\frac{m_1-2}{m_1-1}\,v_1$ is part of a translate of $D_2\backslash\{v_2\}$ contained in $A\cap T$,
and hence also of a segment  $\sigma\subset T$ of length at least 
$\frac{m_2-2}{m_2-1}\,\|v_2\|$.
Since $\frac{m_i-2}{m_i-1}\geq \frac12$,
we deduce that $m_1=m_2=3$, $D_1=\{v_0,a_1,v_1\}$ and $D_2=\{v_0,a_2,v_1\}$ for $a_2=\frac12\,v_2$. It follows by \eqref{uvD} that $a_0=\frac12(v_1+v_2)=a_1+a_2\in A$.

Let $a\in A\backslash (D_1\cup B')$, and let
$L=a+{\rm lin}\{v_1,v_2\}$.  It follows by \eqref{uvD} applied to the edge $[v_0,v_1]$ that $a\in\{p,p-a_1\}$ where $\{p,p-a_1\}\subset A$.
Now $p\not\in(D_2\cup B')$, and hence  applying \eqref{uvD}  to  the edge $[v_0,v_2]$, we conclude that
either $p+a_2\in A$, or $p-a_2\in A$. In other words, either
$[p,p-a_1,p+a_2]\subset [B]\cap L$, or $[p,p-a_1,p-a_2]\subset [B]\cap L$.
Since $[B]\cap L$ is a translate of $\lambda T$ for $\lambda\in(0,1]$ and 
$\{p,p-a_1\}\cap D_1=\emptyset$, 
we deduce that $\lambda=1$, and $\{p,p-a_1,p-a_2\}=\{a_0,a_2,a_1\}$.
Therefore $A$ consists of the vertices of $[B]$, and the midpoints of $T$.
\end{proof}

\section{Critical pairs $(A,B)$ with ${\rm dim}[B]=2$}
\label{secplanar}

In this section,  $A,B$ are finite sets in $\R^2$ satisfying that $A\subset[B]$ and $B$ spans $\R^2$.
Thus the case $k=1$ of Theorem~\ref{AkB} can be written into the form
\begin{equation}
\label{AkB2}
|A+B|\geq 3|A|-3.
\end{equation}
We note that the case when $[B]$ is a triangle is handled in Section~\ref{secsimplex}, and
hence we start with the case when $[B]$ is a  quadrilateral.

\begin{prop}
\label{quadri}
Let  $(A,B)$ be a critical pair with $[B]=[v_0,v_1,v_2,v_3]$  a quadrilateral. Then $A\subset \partial [B]$ and $[B]$ is a trapezoid. Moreover
\begin{enumerate}
\item[(i)] if $B$ consists of the vertices of a paralelogram, then $A$ can be partitioned into pairs of points, each a translate of a pair of consecutive points of $B$;
\item[(ii)] if $B$ has a loaded edge or is not a paralelogram, then $A$ is contained in two parallel edges of $B$, say $e_{1}=[v_0,v_1]$ and $e_{2}=[v_2,v_3]$, and each of $A\cap e_i$ can be partitioned into maximal arithmetic progressions with common difference $w$; in addition, if $\ell (e_2)\le \ell (e_1)$ and $[v_0,v_2]$ is an edge of $[B]$, then $(A\cap e_2)-(v_2-v_0)= A\cap (e_2-(v_2-v_0))$.
\end{enumerate}
\end{prop}

\begin{proof}  Let $B'=\{v_0,v_1,v_2,v_3\}$ be the vertices of $B$, where we may assume that $v_0=o$ and $[0,v_1], [0, v_2]$ are edges of $[B]$.  By Theorem~\ref{equacond}, we may assume that  $B\subset A$ and that  $(A,B')$ is critical.

Suppose that $[B]$ has no pair of parallel sides.  We say that a side of $[B]$ is big if the sum of the angles at the endpoints of the side is less than $\pi$. Since one side out of two opposite sides of  $[B]$ is big, there exists a vertex of  $[B]$ where two big sides meet. Therefore we may assume that $B=\{o,v_1,v_2,t_1v_1+t_2v_2\}$, where
$$
0<t_1\leq t_2<1\mbox{ \ and \ } t_1+t_2>1.
$$
In addition, if $s_1v_1+s_2v_2\in A$ then $s_1+s_2\leq t_1+t_2$.

Let $B_0'=\{o,v_1,v_2\}$. We observe that any coset of $\Z^2=\Z v_1+\Z v_2$ intersects
$[0,1) v_1+[0,1) v_2$ in exactly one point, therefore
no two points of $A\backslash B_0'$ are in the same coset.
 We deduce  that
$$
|A+B_0'|=|A\setminus B_0'|\cdot |B_0'|+|2B_0'|=3(|A|-3)+6=3\cdot|A|-3.
$$

On the other hand, if $s_1v_1+s_2v_2\in A+B_0$ then $s_1+s_2\leq t_1+t_2+1$.
Therefore $2(t_1+t_2)> t_1+t_2+1$ yields
$$
2(t_1v_1+t_2v_2)\in (A+B')\backslash (A+B_0'),
$$
and hence
$|A+B'|>|A+B_0'|=3\cdot|A|-3$, contradicting that $(A,B')$ is
a critical pair. This proves the first part of the statement. 

For (i), suppose that $B=B'$ and $[B]$ is a paralelogram. We may assume that
$[B]=[0,1]^2$. We partition $A=A_0\cup\cdots \cup A_t$ into equivalence classes
according to $\Z^2$, where $A_0=B$. We observe that, for $i>0$, each $A_i$ consists either of one single point or a pair which is a translate of a pair of consecutive vertices in $B$. We have
$$
|A+B|=\sum_{i=0}^k |A_i+B|=3|A_0|-3+\sum_{i=1}^k|A_i+B|=3|A|-3,
$$
which implies $|A_i+B|=3|A_i|$ for each $i=1,\ldots ,k$. This implies that no $A_i$  consists of a single point.

Finally, to prove (ii), we suppose that  $v_1=\lambda (v_3-v_2)$ with $\lambda \ge 1$ and that the  edge $e_1=[o,v_1]$ is loaded if $\lambda =1$. Set $A_i=A\cap e_i$ and $B_i=B\cap e_i$ for $i=1,2$.   Consider equivalence classes of $A$ determined by the cosets of the subgroup $H=\Z v_2+\R v_1$. One equivalence class of $A$ is $A_1\cup A_2$, and the rest are of the form $A\cap l$ for some line $l$ parallel to $[o,v_1]$, and intersecting ${\rm int}[B]$. For such a line $l$ we claim that 
\begin{equation}
\label{trapezoidinner}
|(A\cap l)+B|>3|A\cap l|.
\end{equation}
Indeed, if $\lambda> 1$ then 
$$
|(A\cap l)+B'|=|(A\cap l)\cup ((A\cap l)+v_1)\cup ((A\cap l)+\{v_2,v_3\})|>3|A\cap l|.
$$
If $\lambda =1$ then $|B_1|>2$ by our assumption, and we have $|(A\cap l)+B_1|\ge |A\cap l|+|B_1|-1= |A\cap l|+2$. Hence,
$$
|(A\cap l)+(B_1\cup B_2)|=|(A\cap l)+B_1|+|(A\cap l)+B_2|\ge( |A\cap l|+2)+(2|A\cap l|-1)>3|A\cap l|.
$$

Since $|(A_1\cup A_2)+(B_1\cup B_2)|\geq 3|(A_1\cup A_2)|-3$ by  Theorem~\ref{AkB}, it follows from   (\ref{trapezoidinner}) that $A=A_1\cup A_2$,
and in turn $B\subset A$ implies that $B=B_1\cup B_2$. This shows that both $A$ and $B$ are contained in two parallel lines of the trapezoid. Therefore,
$$
|A+B|=|A_1+B_1|+|A_2+B_2|+|(A_1+B_2)\cup (A_2+B_1)|,
$$
where, by Proposition~\ref{dimone} (i), $|A_i+B_i|\ge 2|A_i|-1$ with equality if and only if $A_i$ is stable with respect to $B_i$. Moreover we have also $|(A_1+B_2)\cup (A_2+B_1)| \ge |A_1|+|A_2|-1$.  Hence, it follows from $|A+B|=3|A|-3$  that  there is equality in the three inequalities above, which implies 
$$
A_2\subseteq A_1+v_2 \mbox{ and } (A_1+v_2)\cap [v_2,v_3]\subseteq A_2,
$$
which together with the other two equalities imply (ii).
 \end{proof}

\begin{lemma}
\label{polygon-trapezoid}
If $(A,B)$ is a critical pair, and $[B]$ is a polygon, then $[B]$
has at most four vertices.
\end{lemma}
\begin{proof}
We suppose that $P=[B]$ is a polygon of at least five vertices, and seek a contradiction.
According to Theorem~\ref{equacond}, we may assume that $P$ is a pentagon, and $B$ consists of the vertices of $P$.
For any vertex $v$ of $P$, let $P_v$ be the convex hull of the other four vertices of $P$.
It follows again by  Theorem~\ref{equacond} that
 $(A\cap P_v,B\cap P_v)$ is a critical pair, as well, and hence Proposition~\ref{quadri} yields that $P_v$ is a trapezoid.

Since the sum of the angles of $P$ is $3\pi$, there exists a side $f$ of $P$ such that
the sum of the angles at the two endpoints is at least $\frac{6\pi}5>\pi$.
Let $e$ be the diagonal of $P$ not meeting $f$, and let $v$ be the vertex not in $e\cup f$.
It follows that $P_v$ is a trapezoid where $e$ and $f$ are parallel, and $\ell(e)>\ell(f)$.
We deduce from Proposition~\ref{quadri} that there exists $x\in A\cap e$ different from the endpoints of $e$.

Now let $w$ be an endpoint of $f$. Since $e$ is a diagonal of $P_w$, we have
$x\in A\cap {\rm int}P_w$. However
Proposition~\ref{quadri} (i) and (ii)  applied to the pair
$(A\cap P_w,B\cap P_w)$  shows that $A\cap {\rm int}P_w=\emptyset$,
which is a contradiction.
\end{proof}

\section{Critical pairs $(A,B)$  where $[B]$ is an iterated pyramid
over a prism over a simplex}
\label{secprism}

Our first statement is a preparation for the proof of Lemma~\ref{Anoinside}.

\begin{lemma}
\label{Ainside}
If  $A,B\subset \R^d$, $d\geq 2$, are finite such that
   $x+A\subset{\rm int}[B]$ for $x\in\R^d$, and
$[B]$ is  a $d$-dimensional prism over a simplex whose vertical sides are parallel, then
$$
|A+B|> (d+1)|A|.
$$
\end{lemma}
\begin{proof} Let $[v_0,\ldots,v_{d-1}]$ and $[w_0,\ldots,w_{d-1}]$ be the facets of $[B]$
such that $[v_i,w_i]$ are the parallel vertical edges for $i=0,\ldots,d-1$, and
\begin{equation}
\label{v0w0max}
\mbox{$\|v_i-w_i\|\leq \|v_0-w_0\|$ for $i=1,\ldots,d-1$}.
\end{equation}
We may assume that
$v_0$ is the origin, and $A\subset{\rm int}[B]$.  Let $\widetilde{B}=\{v_0,w_0,\ldots,v_{d-1},w_{d-1}\}$ be the vertex set of $[B]$, and let  $B'=\{v_1,w_1,\ldots,v_{d-1},w_{d-1}\}$.

We divide $A$ into equivalence classes according to the subgroup
$$
H=\Z v_1+\ldots+\Z v_{d-1}+\R w_0,
$$
and hence adding $\widetilde{B}\subset H$
to different equivalence classes results in disjoint sets.
As $A\subset{\rm int}[B]$ and (\ref{v0w0max}) yield that $A\subset(0,1)v_1+\ldots+(0,1)v_{d-1}+(0,1) w_0$,
any equivalence class is of the form $A\cap l$ for
a line $l$ parallel to $w_0$ and intersecting $A$, and a translate of $A\cap l$ is contained
in $[v_0,w_0]\backslash\{v_0,w_0\}$. Thus Proposition~\ref{dimone} (ii) implies
\begin{equation}
\label{lv0w0}
|(A\cap l)+ \{v_0,w_0\}|\geq
2|A\cap l|.
\end{equation}
We note that
\begin{equation}
\label{lvi}
\mbox{the sets
$l+v_i$, $i=0,\ldots,d-1$, are pairwise disjoint,}
\end{equation}
therefore
\begin{equation}
\label{lviwi}
|(A\cap l)+B'|=\sum_{i=1}^{d-1}|(A\cap l)+\{v_i,w_i\}|\geq
\sum_{i=1}^{d-1}(|A\cap l|+1)>
(d-1)|A\cap l|.
\end{equation}
We conclude $|A+B|\geq |A+\widetilde{B}|> (d+1)|A|$ by 
\eqref{lv0w0}, \eqref{lvi} and \eqref{lviwi}.
\end{proof}

Combining Lemmas~\ref{join} and \ref{Ainside} yield the following.

\begin{lemma}
\label{Anoinside}
If  $d>q\geq 2$ and $(A,B)$ is a critical pair such that  $B$ spans $\R^d$, $A\subset[B]$,
 and $[B]=[x_1,\ldots,x_{d-q}, P]$  for a $q$-dimensional prism over a simplex $P$ whose vertical edges are parallel, then
$A=\{x_1,\ldots,x_{d-q}\}\cup(A\cap P)$.
\end{lemma}

It remains to describe the structure of $A$ and $B$ when $[B]$ is a prism over a simplex. 

\begin{prop}
\label{projectiveprism}
If  $d\geq 3$, and $(A,B)$ is a critical pair such that $[B]$ is  a $d$-dimensional prism over a simplex, then
$B\subset A$, and
\begin{description}
\item{(i)} the vertical edges of $[B]$ are parallel,
\item{(ii)} $A$ is contained in the vertical edges of $[B]$,
\item{(iii)} there exists a vertical vector $w\neq 0$ such that for each vertical edge $e$,
$A\cap e$ can be partitioned into maximal arithmetic progressions of difference  $w$ in $e$,
one of them  containing both endpoints of $e$, and this longest arithmetic progression
contains $B\cap e$. In addition if $e$ and $f$ are vertical edges, and $e+v\subset f$ in a way such that $e+v$ and $f$ share a common endpoint, then $(A\cap e)+v=A\cap(e+v)$.
\end{description}
\end{prop}
\begin{proof} We have $B\subset A$ by Theorem~\ref{equacond} (i).
Let $[v_0,\ldots,v_{d-1}]$ and $[w_0,\ldots,w_{d-1}]$ be the facets of $[B]$
such that $[v_i,w_i]$ are the vertical edges for $i=0,\ldots,d-1$.

For $0\leq i<j\leq d-1$, it follows from Theorem~\ref{equacond} (iii)
that  $(A_{ij},B_{ij})$ is a critical pair
for $A_{ij}=[v_i,w_i,v_j,w_j]\cap A$ and $B_{ij}=[v_i,w_i,v_j,w_j]\cap B$, therefore
Proposition~\ref{quadri}  yields
\begin{eqnarray}
\label{vivjtrap}
[v_i,w_i,v_j,w_j]&&\mbox{is a trapezoid, and}\\
\label{vivjAB}
\mbox{$A_{ij}$ and $B_{ij}$}&&
\mbox{satisfy the conditions of Proposition~\ref{quadri} (i) or
(ii).}
\end{eqnarray}

We verify (i) using an indirect argument. We suppose that the lines
of the vertical edges meet at a point $p\in\R^d$, and seek a contradiction.
We may assume that $w_i\in[p,v_i]$ for $i=0,\ldots,d-1$.
Since the pair composed of
 $A\cap[v_0,v_1,w_1,v_2,w_2]$ and $B\cap[v_0,v_1,w_1,v_2,w_2]$
is critical by  Theorem~\ref{equacond} (iii), and $[v_1,w_1,v_2,w_2]$ is a trapezoid according to
(\ref{vivjtrap}) ,  Lemma~\ref{Anoinside} yields that
\begin{equation}
\label{emptyedge}
A\cap[v_0,v_1]=\{v_0,v_1\}.
\end{equation}
However $w_i\in[p,v_i]$ for $i=0,1$ and (\ref{vivjtrap}) yield that $[w_0,w_1]$ is parallel
with and shorter than $[v_0,v_1]$. We deduce from Proposition~\ref{quadri}
that $|A\cap[v_0,v_1]|\geq 3$, contradicting (\ref{emptyedge}), and implying (i).

We prove (ii) again by contradiction, therefore we suppose that
there exists an $x\in A$ not contained in the vertical edges. 

According to the Charateodory theorem (see e.g. B. Gr\"unbaum \cite{Gru03}), if $x\in[X]\backslash X$ for $X\subset\R^d$, then
$x\in[x_0,\ldots,x_d]$ for $x_0,\ldots,x_d\in X$. It follows that possibly after
reindexing, there exists $m$ and such that $1\leq m\leq d$,
$[x_0,\ldots,x_m]$ is an $m$-simplex, and
\begin{equation}
\label{Charateodory}
x\in{\rm relint}[x_0,\ldots,x_m]
=\{t_0x_0+\ldots+t_mx_m:\,\mbox{$t_0,\ldots,t_m>0$ and
$t_0+\ldots+t_m=1$}\}.
\end{equation}

We deduce from (\ref{Charateodory}) that there exists $1\leq m\leq d$,
and affinely independent vertices $x_0,\ldots,x_m$ of $[B]$ such that 
$x\in{\rm relint} [x_0,\ldots,x_m]$. Since an $m$-simplex has no two parallel edges, we may assume
that $x_0=v_0$, and $w_0\not\in \{x_1,\ldots,x_m\}$. In particular,
$x\in A\cap P$ for $P=[v_0,v_1,w_1,\ldots,v_{d-1},w_{d-1}]$ where
$Q=[v_1,w_1,\ldots,v_{d-1},w_{d-1}]$ is a  prism over a simplex  with parallel vertical edges
by (i). It follows from $x\in{\rm relint} [x_0,\ldots,x_m]$ that
$x\neq v_0$ and $x\not\in Q$.
Since the pair $(A\cap P, B\cap P)$ is critical, this
contradicts Lemma~\ref{Anoinside}, and hence implies  (ii).

The last property (iii) follows from (\ref{vivjAB}) and Proposition~\ref{quadri}.
\end{proof}

\section{Proof of necessity in Theorem~\ref{AkBequalist} }
\label{secproof}

Let $(A,B)$ be a $k$--critical pair for some $k\geq 1$ with ${\rm dim}[B]=d$.
In particular, $(A,B)$ a $1$--critical and $B$  is totally stackable by Theorem~\ref{equacond}. According to Theorem~\ref{splitchar}, $[B]$ is a simplex, or over an iterated pyramid over a
polygon or a prism over a simplex. If $[B]$ is a simplex, then the characterization in
Theorem~\ref{AkBequalist} (i), (ii) and (iii)  is achieved by Proposition~\ref{dimone} if $d=1$, and Proposition~\ref{simplex} if $d\geq 2$.

Therefore let $[B]$ be an iterated pyramid over $P$ with ${\rm dim}P=q$ where $P$ is
polygon or a prism over a simplex. We may assume that $P$ is not a triangle.
Since the pair $(A\cap P,B\cap P)$ is $1$--critical by Theorem~\ref{equacond} (ii),
Proposition~\ref{quadri} and \ref{polygon-trapezoid} yield that if $P$ is a polygon,
then it is a trapezoid. In addition, Proposition~\ref{projectiveprism}
yields that the vertical edges of $P$ are parallel even if $q\geq 3$.

We deduce from Lemma~\ref{Anoinside}
that any point of $A$ is a vertex of $[B]$, or contained in $P$.
Therefore we conclude Theorem~\ref{AkBequalist} (iv) and (v) from
Proposition~\ref{quadri}  if $q=2$, and from
 Proposition~\ref{projectiveprism} if $q\geq 3$.

\noindent {\bf Acknowledgement } Part of the research was done during an FP7 Marie Curie Fellowship of the first name author at BarcelonaTech, whose hospitality is gratefully acknowledged..

K. B\"or\"oczky :  Alfr\'ed R\'enyi Institute of Mathematics,
	 Hungarian Academy of Sciences,
	 1053 Budapest, Re\'altanoda u. 13-15.
	 HUNGARY. and

\noindent Central European University,
1051 Budapest, Nador u. 9, HUNGARY

F. Santos: Departamento de Matem\'aticas, Estad\'{\i}stica y Computaci\'on
Universidad de Cantabria, Av. de los Castros 48
E-39005 Santander, SPAIN
	 
O. Serra: Dept. Matem\`atica Aplicada 4, Univ. Polit\`ecnica de Catalunya, Jordi Girona 1, E-08034, Barcelona, SPAIN

\end{document}